\documentclass[12pt,reqno]{amsart}

\usepackage{amssymb}

\newtheorem{theorem}{Theorem}[section]
\newtheorem{proposition}[theorem]{Proposition}
\newtheorem{lemma}[theorem]{Lemma}
\newtheorem{corollary}[theorem]{Corollary}
\newtheorem{question}[theorem]{Question}

\theoremstyle{definition}
\newtheorem{definition}[theorem]{Definition}
\newtheorem{remark}[theorem]{Remark}

\numberwithin{equation}{section}

\begin{document}

\baselineskip=15.5pt

\title[Fujiki class $\mathcal C$ and geometric structures]{Fujiki class $\mathcal C$ 
and holomorphic geometric structures}

\author[I. Biswas]{Indranil Biswas}

\address{School of Mathematics, Tata Institute of Fundamental
Research, Homi Bhabha Road, Mumbai 400005, India}

\email{indranil@math.tifr.res.in}

\author[S. Dumitrescu]{Sorin Dumitrescu}

\address{Universit\'e C\^ote d'Azur, CNRS, LJAD, France}

\email{dumitres@unice.fr}

\subjclass[2010]{53B35, 53C55, 53A55}

\keywords{Holomorphic geometric structure; Fujiki class $\mathcal C$; algebraic 
reduction; torus bundle.}

\date{}

\begin{abstract}
For compact complex manifolds with vanishing first Chern class that are compact torus principal bundles over K\"ahler manifolds, we prove that
all holomorphic geometric structures on them, of affine type, are locally homogeneous.
For a compact simply connected complex manifold in Fujiki class $\mathcal C$,
whose dimension is strictly larger than the algebraic dimension, we prove that it does
not admit any holomorphic rigid geometric structure, and also it
does not admit any holomorphic Cartan geometry 
of algebraic type. We prove that compact complex simply connected manifolds in 
Fujiki class $\mathcal {C}$ and with vanishing first Chern class do not admit any holomorphic Cartan geometry of algebraic type.
\end{abstract}

\maketitle

\section{Introduction}

The article deals, in particular, with holomorphic geometric structures, in the sense of
Gromov \cite{Gr, DG}, on compact 
complex manifolds. The definition is very general (see Section \ref{section: geometric 
structures}), and interesting classical examples of such structures
to keep in mind are holomorphic tensors, 
holomorphic affine connections, holomorphic projective connections and holomorphic conformal 
structures. Compact complex manifolds bearing any of these kind of geometric structures are 
rather special. Based on results obtained in \cite{D3, BD} we formulate the following:

\begin{question}\label{q1}
Is it true that any holomorphic geometric structure of affine type on 
any compact complex manifold with trivial canonical bundle is locally homogeneous?
\end{question}

Question \ref{q1} is known to have a positive answer if the manifold is either K\"ahler (hence Calabi-Yau) \cite{D3}, or 
if it is in Fujiki class ${\mathcal C}$ with a polystable holomorphic tangent bundle (with respect to some Gauduchon 
metric) \cite{BD}. The results of \cite{D3} show that the answer is also yes for compact parallelizable manifolds and for 
Ghys's deformations of the complex structure on the parallelizable manifolds ${\rm SL(2, \mathbb C)} / \Gamma$ with 
$\Gamma$ being a uniform lattice \cite{Gh2}.

Here we prove that Question \ref{q1} has a positive answer for compact complex torus 
principal bundles over compact K\"ahler Calabi-Yau manifolds (Theorem \ref{fibration}).

Since compact complex surfaces with trivial canonical bundle are either complex tori, or 
K3 surfaces, or primary Kodaira surfaces (elliptic principal bundles over elliptic curves) 
\cite[Chapter~6]{BHPV}, from Theorem \ref{fibration} it follows that the answer to
Question \ref{q1} is yes when the dimension of the manifold is two.

More precisely, our result in this direction is:

\begin{theorem}\label{fibration}
Let $X$ be a compact complex torus holomorphic principal bundle over a compact K\"ahler 
manifold with trivial first Chern class in $H^2(X,\, \mathbb R)$. Then any holomorphic 
geometric structure of affine type $\phi$ on $X$ is locally homogeneous.
\end{theorem}

The previous result also stands for holomorphic projective connections and for holomorphic conformal structures even 
though these two geometric structures are not of affine type. This is because on manifolds with trivial canonical bundle 
these two geometric structures admit global representatives which are of affine type, namely, a holomorphic affine
connection and a holomorphic Riemannian metric respectively.

On the other hand, the previous result does not work in general for non-affine geometric 
structures. To give an example, recall that compact complex tori $T^n\,=\,\mathbb C^n / \Lambda$ (with 
$\Lambda$ a cocompact lattice) admit holomorphic foliations that are defined by nonconstant holomorphic 
maps into the complex projective space ${\mathbb C}P^{n-1}$ which are not translation 
invariant \cite{Gh}. Together with the standard holomorphic parallelization of the 
holomorphic tangent bundle of $T^n$ they form holomorphic {\it rigid} geometric 
structures of non-affine type in Gromov's sense (see Definition \ref{def-a} and
Definition \ref{d-rigid}) which are not locally homogeneous.

We expect the answer for Question \ref{q1} to be yes for all holomorphic geometric 
structures of affine type on compact complex manifolds in the Fujiki class $\mathcal 
C$ that have trivial canonical bundle. We also expect holomorphic Riemannian metrics to 
be always locally homogeneous on compact complex manifolds (this was proved in complex 
dimension three \cite{D4}) and holomorphic affine connections to be always locally 
homogeneous on compact complex manifolds with trivial canonical bundle; it may be noted that contrary to 
the case of holomorphic Riemannian metrics, here the condition on the triviality of 
the canonical bundle is not automatically satisfied and is in fact even necessary: there exists 
non-locally homogeneous holomorphic affine connections on principal elliptic bundles 
with odd first Betti number (hence non-K\"ahler) over Riemann surfaces of genus $g 
\,\geq\, 2$ \cite{D5}.

In the case where the holomorphic geometric structure $\phi$ is rigid, Theorem 
\ref{fibration} enables us to gather information about the fundamental group of the manifold 
$X$:

\begin{corollary}\label{main corollary}
Let $X$ be a compact complex torus holomorphic 
principal bundle over a compact K\"ahler manifold with trivial first Chern class. If
$X$ is endowed with a holomorphic rigid geometric structure of affine type $\phi$, 
then the fundamental group of $X$ is infinite.
\end{corollary}

\begin{corollary}\label{second corollary}
Let $X$ be a compact complex manifold in 
Fujiki class $\mathcal C $ bearing a holomorphic affine connection in $TX$. Then the 
fundamental group of $X$ is infinite.
\end{corollary}

Notice that Theorem \ref{fibration} applies to many complex torus principal bundles which are not K\"ahler. For example, 
one could consider the complex Heisenberg group of upper triangular $(3 \times 3)$ matrices with complex entries and 
the compact parallelizable manifold obtained by taking the 
quotient of this group by the lattice of matrices with Gaussian integers 
as entries. This quotient is biholomorphic to a (non-K\"ahler) principal elliptic bundle over a two-dimensional compact 
complex torus. Many other examples and results about complex torus principal bundles can be found in \cite{Hof}.

Recall that complex projective spaces admit the standard flat holomorphic projective 
connection. Also recall that the smooth quadric \(z_0^2+ z_1 ^2 + \ldots 
+z_{n+1}^2=0\) in ${\mathbb C}P^{n+1}$ is endowed with a canonical (flat) holomorphic 
conformal structure given by the quadratic form (this standard holomorphic conformal 
structure is invariant by the subgroup ${\rm PO}(n+2, \mathbb{C})$ of the complex 
projective group $\text{PGL}(n+2,{\mathbb C})$). More generally, if $G$ is a complex 
semi-simple Lie group and $P\, \subset\, G$ a parabolic subgroup, then the rational manifold $G/P$ 
is equipped with the standard flat Cartan geometry with model $(G,P)$.

For holomorphic rigid geometric structures and for holomorphic Cartan geometries of algebraic 
type (definitions are given in Section \ref{section: geometric structures}) we prove the 
following:

\medskip
\noindent
\text{Theorem \ref{dim less 2}. }~ {\it Let $X$ be a compact complex manifold in the Fujiki 
class $\mathcal{C}$, of complex dimension $n$ and of algebraic dimension $n-d$, with $d>0$. 
If $X$ admits a holomorphic rigid geometric structure $\phi$, then the fundamental group of 
$X$ is infinite.}

\medskip
\noindent
\text{Theorem \ref{Cartan geom}. }~ {\it Let $X$ be a compact complex simply connected manifold in the
Fujiki class $\mathcal{C}$. If $X$ bears a holomorphic Cartan geometry of algebraic type, then $X$ is projective.}
\medskip

Theorem \ref{dim less 2} and Theorem \ref{Cartan geom} are in the same spirit as Borel-Remmert result asserting that 
compact simply connected homogeneous manifolds in Fujiki class $\mathcal C$ are 
projective. Indeed, it should be mentioned that it was conjectured that compact 
simply connected complex manifolds bearing holomorphic Cartan geometries are 
homogeneous manifolds. Borel-Remmert theorem states precisely that compact homogeneous 
manifolds in class $\mathcal C$ are biholomorphic to a product of a projective 
rational homogeneous manifold with a complex torus \cite{Fu} (p. 255).

In the special case where the algebraic dimension of $X$ is zero, Theorem \ref{Cartan geom} 
asserts that $X$ does not bear holomorphic Cartan geometries of algebraic type. This was 
proved recently in \cite{BDM} (Theorem 4.1) even for manifolds which are not necessarily in 
class $\mathcal C$.

Theorem \ref{Cartan geom} implies the following:

\begin{corollary}\label{third corollary}
Let $X$ be a compact complex manifold in 
Fujiki class $\mathcal C $ with trivial first Chern class in 
$H^2(X,\, \mathbb R)$. If $X$ bears a holomorphic Cartan geometry of algebraic type, then the 
fundamental group of $X$ is infinite.
\end{corollary}

Notice that Corollary \ref{third corollary} should be seen as a natural generalization of 
Corollary \ref{second corollary}. Indeed, manifolds $X$ in Fujiki class $\mathcal C$ bearing 
holomorphic affine connections on $TX$ have vanishing Chern classes (see proof of Corollary 
\ref{second corollary} in Section \ref{section:trivial canonical bundle}).

\section{Geometric structures and symmetries} \label{section: geometric structures}

Let $X$ be a complex manifold of complex dimension $n$. For any integer $k \,\geq 
\,1$, we associate the principal bundle of $k$-frames $$R^k(X) \,\longrightarrow\, X\, ,$$ 
which is the bundle of $k$-jets of local holomorphic coordinates on $X$. The 
corresponding structural group $D^k$ is the group of $k$-jets of local biholomorphisms 
of $\mathbb{C}^n$ fixing the origin. We note that $D^k$ is a complex
affine algebraic group.

\begin{definition}\label{def-a}
A {\it holomorphic geometric structure} of order $k$
on $X$ is a holomorphic $D^k$-equivariant map $\phi$ from $R^k(X)$ to a
complex algebraic manifold $Z$ endowed with an algebraic action of $D^k$.
The geometric structure $\phi$ is said to be of affine type if $Z$ is a complex affine manifold.
\end{definition}

Holomorphic tensors are holomorphic geometric structures of affine type of order one. 
More precisely, a holomorphic tensor on $X$ is a holomorphic ${\rm GL}(n, \mathbb 
C)$-equivariant map from the frame bundle $R^1(X)$ to a linear complex algebraic 
representation $W$ of ${\rm GL}(n, \mathbb C)$. Holomorphic affine connections are 
holomorphic geometric structures of affine type of order two \cite{Gr,DG}. Holomorphic 
foliations and holomorphic projective connections are holomorphic geometric structure 
of non-affine type.

The natural notion of symmetry of a holomorphic geometric structure is the following. 
A (local) biholomorphism of $X$ preserves a holomorphic geometric structure $\phi$ if 
its canonical lift to $R^k(X)$ fixes each fiber of the map $\phi$. Such a local biholomorphism 
is called a {\it local isometry} of $\phi$.

A (local) holomorphic vector field on $X$ is called a {\it Killing vector field} with 
respect to $\phi$ if its local flow acts on $X$ by local isometries.

The connected 
component ${\rm Aut}_0(X, \phi)$, containing the identity element, of the automorphism group of $(X,\, \phi)$ is
a complex Lie subgroup of the automorphism group of $X$. The corresponding Lie algebra is the vector 
space of globally defined (holomorphic) Killing vector fields for $\phi$.

\begin{definition}\label{d-rigid}
A holomorphic geometric structure $\phi$ is called {\it rigid} of order $l$ in 
Gromov's sense if any local biholomorphism preserving $\phi$ is uniquely determined by its 
$l$-jet in any given point.
\end{definition}

Holomorphic affine connections are rigid of order one in Gromov's sense (see \cite{Gr} 
and the nice expository survey \cite{DG}). The rigidity comes from the fact that local 
biholomorphisms fixing a point and preserving a connection linearize in exponential 
coordinates, so they are indeed completely determined by their differential at the fixed 
point.

Holomorphic Riemannian metrics, holomorphic projective connections and holomorphic 
conformal structures in dimension $\geq 3$ are all rigid holomorphic geometric structures,
while holomorphic symplectic structures and holomorphic foliations are non-rigid geometric 
structures~\cite{DG}.

The sheaf of local Killing fields of a holomorphic rigid geometric structure $\phi$ is 
locally constant. Its fiber is a finite dimensional Lie algebra called {\it the 
Killing algebra} of $\phi$ \cite{DG,Gr}.

Gromov's study of local symmetries of analytic rigid geometric structure \cite{DG, Gr} 
led, in the particular case of simply connected manifolds $X$, to the following 
description of ${\rm Aut}_0(X, \phi)$ (see Section 3.5 in \cite{Gr}). The action of 
${\rm Aut}_0(X, \phi)$ preserves each connected component of the fibers of a meromorphic map
$$\phi^{r} \,:\, X \,
\longrightarrow\, W^r$$ into an algebraic manifold $W^r$ (representing the $r$-jets of $\phi$), and furthermore,
the action of ${\rm Aut}_0(X, \phi)$ on each of these connected components is transitive.
The two main ingredients of the proof are
\begin{enumerate}
\item the integrability result showing that, for any $r$ large 
enough, local isometries are exactly the class of local biholomorphisms that preserve the $r$-jet of $\phi$, and

\item the 
extendibility result proving that local Killing fields on simply connected manifolds 
extend to the entire manifold \cite{Am,DG,Gr, No}.
\end{enumerate}
It now follows that ${\rm Aut}_0(X, 
\phi)$-orbits in $X$ are locally closed as they coincide with the connected components 
of the fibers of the meromorphic map $\phi^{r}$.

Inspiration of these results led to Theorem 2.1 in \cite{D1} which 
says that the Killing Lie algebra of $\phi$ act transitively on the connected 
components of the {\it algebraic reduction} of $X$ (see also Theorem 3 in \cite{D3}).

The algebraic dimension $a(X)$ of $X$ is the transcendence degree over $\mathbb C$ of 
the field of meromorphic functions ${\mathcal M} (X)$.

Let us recall the following classical
result called the {\it algebraic reduction} theorem (see \cite[pp. 25--26]{Ue}). 

\begin{theorem}[{Algebraic Reduction, \cite[p.~25, Definition~3.3]{Ue}, \cite[p.~26, Proposition~3.4]{Ue}}]\label{thue}
Let $X$ be a compact connected complex manifold of dimension $n$ and algebraic dimension $a(X)\,=\,n-d$. There
exists a bi-meromorphic modification $$\Psi \,:\, \widetilde{X}\,\longrightarrow\, X$$ and a holomorphic map
$$t \,:\, \widetilde{X}\,\longrightarrow\, V$$ with connected fibers onto a $(n-d)$-dimensional projective manifold $V$ such
that $$t^* ({\mathcal M}(V))\,=\, \Psi^*({\mathcal M} (X))\, .$$
\end{theorem} 

Let $\pi_{red}\,:\, X \,\longrightarrow\, V$ be the meromorphic fibration given by 
$t \circ \Psi^{-1}$; it is called the algebraic reduction of $X$.

This meromorphic fibration $\pi_{red}$ is called {\it almost holomorphic} if the 
$\Psi$-exceptional locus does not the intersect generic $t$-fibers.

Manifolds with maximal algebraic dimension are those for which the algebraic 
dimension coincides with the complex dimension. They are called Moishezon 
manifolds \cite[p.~26, Definition~3.5]{Ue}. The algebraic reduction of a Moishezon manifold is a bimeromorphism with 
a smooth complex projective manifold \cite{Mo}, \cite[p.~26, Theorem~3.6]{Ue}.

More generally, a compact complex manifold is said to be in {\it the Fujiki class 
$\mathcal C$} if it is the image of a compact K\"ahler space under a holomorphic map.
A result of Varouchas says
that a compact complex manifold belongs to Fujiki's class $\mathcal C$ if and
only if it is bimeromorphic to a compact K{\"a}hler manifold
(in other words, admits compact K\"ahler modifications) \cite[Section\,IV.3]{Va}.
Manifolds in Fujiki's class $\mathcal C$ share many 
common features with K\"ahler manifolds.

We will investigate in Section \ref{seF} the holomorphic geometric structures on 
compact complex manifolds in Fujiki class $\mathcal C$.

The following theorem is proved using Theorem 2.1 in \cite{D1}.

\begin{theorem}\label{act alg red} Let $X$ be a compact complex simply connected 
manifold of dimension $n$ endowed with a holomorphic rigid geometric structure. Then there exists a 
connected complex abelian Lie subgroup $L$ in the group of automorphisms of $(X,\, 
\phi)$ which preserves each fiber of the algebraic reduction $\pi_{red}$, and which acts 
transitively on the generic fibers of $\pi_{red}$. Moreover the following hold.
\begin{enumerate}
\item[(i)] If $\pi_{red}$ is almost holomorphic, then
the $L$-orbits are compact. 

\item[(ii)] If $X$ is in the Fujiki class $\mathcal C$, then $a(X)\,>\,0$.
\end{enumerate}
\end{theorem}

\begin{proof} 
It follows from Theorem 2.1 in \cite{D1} that the orbits of ${\rm Aut}_0(X, \phi)$ contain the generic orbits of 
$\pi_{red}$ (see also Theorem 3 in \cite{D3}). Let $\{X_1,\, \cdots,\, X_l\}$ be a basis of the Lie algebra of ${\rm 
Aut}_0(X, \phi)$; these $X_i$ are globally defined holomorphic vector fields on $X$.

Consider the holomorphic rigid geometric structure $\phi'$ which is the juxtaposition of $\phi$ with the family of 
holomorphic vector fields $\{X_1,\, \cdots,\, X_l\}$. Denote by $L'$ the automorphism group ${\rm Aut}_0(X, \phi')$ of 
$\phi'$. It can be shown that the
Lie algebra of $L'\,=\, {\rm Aut}_0(X, \phi')$ is in fact the center of the Lie algebra of ${\rm 
Aut}_0(X, \phi)$. Indeed, an element of the family $\{X_1,\, \cdots,\, X_l\}$ preserves $\phi'$ if and only if it commutes 
with all elements in this family (it belongs to the center of the Lie algebra). It follows that $L'$ is the maximal 
abelian connected subgroup in ${\rm Aut}_0(X, \phi)$.

Applying Theorem 2.1 in \cite{D1} to $\phi'$ it follows that at generic points in $X$, the Killing Lie algebra of 
$\phi'$ (and hence $L'$) acts transitively on the fibers of the algebraic reduction $\pi_{red}$. 

Let us consider the almost rigid meromorphic geometric structure (in the sense of Definition 1.2 in \cite{D1}) $\phi''$ 
obtained by juxtaposing $\phi$, the family of holomorphic vector fields $\{X_1,\, \cdots,\, X_l\}$ and the meromorphic map 
$\pi_{red}$. When applied to $\phi''$, Theorem 2.1 in \cite{D1} says that, at generic points, the Killing Lie algebra of 
$\phi''$ acts transitively on the fibers of the algebraic reduction $\pi_{red}$. Notice that Killing
vector fields of $\phi''$ are 
exactly those Killing vector fields of $\phi$ lying in the center of the Killing Lie algebra of $\phi$ and preserving each fiber 
of $\pi_{red}$. The connected component of identity in the automorphism group ${\rm Aut}(X, \phi'')$ is the maximal 
connected complex abelian Lie subgroup $L\,=\,{\rm Aut}_0(X, \phi'')$ of ${\rm Aut}_0(X, \phi)$ (and hence of $L'$) preserving each fiber of the algebraic reduction.

(i) If $\pi_{red}$ is almost holomorphic, its generic fibers are compact. Therefore the generic (and hence all) $L$-orbits are compact.

(ii) If $X$ is in the Fujiki class $\mathcal C$, the connected component of 
identity ${\rm Aut}_0(X)$ of the automorphism group of the (simply connected) 
manifold $X$ is a complex linear algebraic group (see Corollary 5.8 in \cite{Fu}). 
Then ${\rm Aut}_0(X, \phi)$ is the connected component of identity of the subgroup 
of ${\rm Aut}_0(X)$ preserving each fiber of the meromorphic fibration 
$$\phi^{r}\,:\, X \, \longrightarrow\, W^r\, .$$ In particular, ${\rm Aut}_0(X, \phi)$ 
is closed in ${\rm Aut}_0(X)$. Moreover, $L\,=\,{\rm Aut}_0(X, \phi'')$ coincides 
with the connected component of identity of the maximal abelian subgroup of ${\rm 
Aut}_0(X, \phi)$ preserving each fiber of the algebraic reduction of $X$. It 
follows that $L$ is closed in ${\rm Aut}_0(X, \phi)$. A priori $L$ is not Zariski closed with respect to the complex 
linear algebraic group structure of ${\rm Aut}_0(X)$. Its Zariski closure $L^*$ is an algebraic
abelian subgroup of ${\rm Aut}_0(X)$ in the terminology of Fujiki (see Definition 2.1 in 
\cite{Fu}), so $L^*$ is a meromorphic subgroup of ${\rm Aut}_0(X)$ (equivalently, the 
$L^*$-action on $X$ is compactifiable in Lieberman's terminology \cite[Section~3]{Li}).

Since $L^*$ is a connected complex abelian algebraic group, a classical result of Rosenlicht \cite{Ro} shows that $L^*$ is 
isomorphic to ${\mathbb C}^p \times {{\mathbb C}^*}^q$ for some nonnegative integers 
$p$ and $q$.

Assume, by contradiction, that $a(X)\,=\,0$. Then $L$ acts with an open dense orbit on $X$. Consequently, $L^*$ acts with 
an open dense orbit on $X$. Since $L^*$ is abelian, this open dense orbit is biholomorphic to $L^*$ and $X$ is a 
compactification of $L^*$.

Since $L^*$ is algebraic, it can be seen as a Zariski open dense set in a complex rational manifold $L^{**}$ \cite{Ro}. 
The action of $L^*$ on $X$ is meromorphic in Fujiki's sense,
meaning the holomorphic action map $$L^* \times X \,\longrightarrow\, X$$ extends to a meromorphic map $L^{**} \times X
\,\longrightarrow\, X$ (see \cite{Fu}, Proposition 2.2 and Remark 2.3). Therefore, $X$ is a bi-meromorphic image of
$L^{**}$ and hence it is a projective rational manifold (see Lemma 3.8 and Remark 4.1 in \cite{Fu}). It follows that
$a(X)\,=\,n \,>\,0$: a contradiction.
\end{proof}

Notice that statement (ii) in Theorem \ref{act alg red} will be improved in Theorem 
\ref{dim less 2} which asserts that simply connected manifolds in Fujiki class 
$\mathcal C$ bearing holomorphic rigid geometric structures are Moishezon. 

The following classical result will be useful in the proof of Proposition \ref{dim less}.

\begin{lemma}\label{vol}
Let $X$ be a compact complex manifold with trivial canonical line bundle $K_X$, and let $L$ be a connected 
complex Lie group acting holomorphically on $X$. Then $L$ preserves any nonzero holomorphic section $\omega$ of $K_{X}$ and, 
consequently, the smooth finite measure on $X$ defined by the section $\omega \wedge \overline{\omega}$.
\end{lemma}

\begin{proof} Let $K$ be a holomorphic fundamental vector field of the $L$-action on $X$. Consider the corresponding
$1$-parameter family of automorphisms of $X$. The automorphism for any $t \,\in\, {\mathbb C}$
will be denoted by $\Phi^t$. The Lie derivative $L_K \omega$ of $\omega$ with respect to $K$ is 
another holomorphic section of $K_{X}$ which must be of the form $c \omega$, for some $c \,\in\, \mathbb C$ (also called the 
divergence of $K$ with respect to $\omega$), because $K_X$ is trivial. This implies that $(\Phi^t)^{*} \omega \,=\, \exp(ct) \cdot \omega$, for all 
$t \,\in\, \mathbb C$. Since the total volume $\int_X \omega \wedge \overline{\omega}$ of $X$ is invariant by automorphisms, it 
follows that the modulus of $\exp(ct)$ equals one. By Liouville Theorem the entire holomorphic map $t \longrightarrow 
\exp(ct)$ must be constant, equal to $1$ (the value in $t\,=\,0$). Consequently, we have $c\,=\,0$ and $\omega$ is $K$-invariant. Since 
$L$ is connected, it is generated by the flows of its fundamental vector fields. Since all of them preserve $\omega$, we 
get that $\omega$ is $L$-invariant.
\end{proof}

In the last section we will deal also with another concept of geometric structure: {\it the Cartan geometry}. Cartan 
geometries are geometric structures which are infinitesimally modeled on homogeneous spaces $G/H$, where $G$ is a complex 
Lie group and $H\, \subset\, G$ is a closed subgroup.

Denote by $\mathfrak{g}$ and $\mathfrak h$ the Lie algebras of the Lie groups $G$ 
and $H$ respectively. Then we have the following definition (see \cite{Sh}).

\begin{definition}\label{cartan geom def}
A {\it holomorphic Cartan geometry} $(P,\, \omega)$ on $X$ with model $(G,\,H)$ is a 
holomorphic principal (right) $H$-bundle $$\pi\,:\, P \,\longrightarrow\, X$$ endowed with a holomorphic $\mathfrak g$-valued
1-form $\omega$ satisfying:
\begin{enumerate}
\item $\omega_{p} \,:\, T_{p}P \,\longrightarrow\, \mathfrak g$ is a complex linear isomorphism for all $p \,\in\, P$,
 
\item the restriction of $\omega$ to every fiber of $\pi$ coincides with the left invariant Maurer-Cartan form of $H$, and
 
\item $(R_{h})^* \omega \,=\,{\rm Ad}(h)^{-1} \omega$, for all $h \,\in \,H$, where $R_{h}$ is the right action of $h$ on $P$
and ${\rm Ad}$ is the adjoint representation of $G$ on $\mathfrak g$.
\end{enumerate}
\end{definition}

The Cartan geometry is said to be {\it of algebraic type} if the image of $H$ through the adjoint representation of $G$ is 
a complex algebraic subgroup of ${\rm GL}(\mathfrak{g})$.

Usual geometric structures such as holomorphic affine connections, holomorphic projective connections, or holomorphic conformal 
are examples of holomorphic Cartan geometry of algebraic type \cite{Sh, Pe}.

A (local) holomorphic vector field on $X$ is a {\it Killing field} for the Cartan geometry $(P,\, \omega )$ if it admits a 
lift to $P$ whose (local) flow commutes with the action of $H$ and also preserves $\omega$. The sheaf of local Killing fields of 
a holomorphic Cartan geometry is locally constant with fiber a finite dimensional Lie algebra \cite{Me,Pe}: the Killing 
algebra of $(P,\, \omega)$.

Gromov's results describing the decomposition in orbits of the Killing algebra of a rigid geometric structure \cite{Gr} 
was adapted to the context of Cartan geometries of algebraic type by Melnick and Pecastaing \cite{Me, Pe}. Using their 
results, Theorem 1.2 in \cite{D2} adapts Theorem 2.1 in \cite{D1} to Cartan geometries of algebraic type. Therefore, 
Theorem \ref{act alg red} stands also for holomorphic Cartan geometries of algebraic type.

\section{Trivial canonical bundle and geometric structures}\label{section:trivial canonical bundle}

In this section we prove Theorem \ref{fibration}, and we deduce Corollary \ref{main 
corollary} and Corollary \ref{second corollary}. We also prove Proposition \ref{prop: weak 
version} which is a particular case of Theorem \ref{dim less} (the case where the canonical 
bundle is trivial).

Let us first prove Theorem \ref{fibration}.

\begin{proof}[{Proof of Theorem \ref{fibration}}] Let $T$ be a complex torus of complex 
dimension $d \,>\,0$, and let $$\pi\,:\, X \,\longrightarrow\, Y$$ be a principal 
$T$-fibration over a compact K\"ahler manifold $Y$. Since the torus action trivializes the 
relative tangent bundle of the fibration, the exact sequence induced by $d \pi$ yields $K_X 
\,\simeq\, \pi^* K_Y$.

We also assume that
\begin{itemize}
\item the real first Chern class 
of $K_X$ vanishes (or equivalently, the real first Chern class of $K_Y$ vanishes), and

\item $X$ 
is endowed with a holomorphic geometric structure $\phi$ of order $k$ which is of affine type.
\end{itemize}
We need to consider only the case 
where the complex dimension $n-d$ of $Y$ is positive. Otherwise, $X$ coincides with the complex torus $T$ and, since 
$R^k(T)$ is isomorphic to $D^k \times T$, the geometric structure $\phi$ is completely determined by a holomorphic map 
from $T$ to a complex affine manifold. This map must be constant, and consequently, $\phi$ is invariant under the action
of $T$ on $X$.

Since $Y$ is K\"ahler with vanishing first Chern class, its canonical bundle $K_Y$ is of finite order \cite{Be}. Replacing
$Y$ by a suitable finite unramified cover of it we will assume that $K_Y$ is trivial (equivalently, $K_X$ is trivial).

Denote by $X_1,\, \cdots,\, X_d$ a basis of the fundamental vector fields of the $T$-action on $X$; they are holomorphic global 
vector fields on $X$ which span (and trivialize) the kernel of $d \pi$ (the vertical subbundle $V \,\subset\, TX$).

Assume, by contradiction, that $\phi$ is not locally homogeneous on $X$. Then Lemma 3.2 in 
\cite{D3} proves that, for some integers $a,b \geq 0$, there exists a nontrivial holomorphic 
section of $(TX)^{\otimes a} \otimes( (TX)^*)^{\otimes b}$ which vanishes at some point in 
$X$. Since $K_X$ is trivial, there is a canonical contraction isomorphism between $TX$ and 
$\Lambda^{n-1} (T^*X)$. Using this isomorphism, we get a nontrivial holomorphic section 
section $t$ of $(T^*X)^{\otimes m}$, with $m\,=\,(n-1)a+b$, such that $t$ vanishes at some 
point in $x_0 \in X$.

We will get a contradiction using an induction on the power $m$.

First consider the case of $m\,=\,1$. In this case $t$ is a holomorphic one form. Note that 
$t(X_i)$ is a holomorphic function on $X$ vanishing at $x_0$ and therefore, it vanishes 
identically. This proves that $V$ is in the kernel of $t$. The fibers of $\pi$ being compact 
and connected, the form $t$ is the pull-back of a holomorphic one form on $Y$. This leads to a 
holomorphic one form on $Y$ vanishing at $\pi(x_0) \,\in \,Y$. Since $Y$ is a K\"ahler 
Calabi-Yau manifold, holomorphic tensors on $Y$ are parallel with respect to any Ricci flat 
K\"ahler metric on it \cite[p.~760, Principe~de~Bochner]{Be}. Hence they cannot vanish at a 
point without being trivial: a contradiction.

Let us consider now the case of $m\,>\,1$. The induction hypothesis is that the vanishing holds if degree of the tensor
is less than $m$.

We will show that for all $0 \,<\, k \,\leq\, m$, the contraction of $t$ with any ordered family of $k$ vector fields (they
need not be distinct), chosen from the vector fields $\{X_1,\, \cdots,\, X_d\}$, vanishes identically.

Let us consider first the case of $k\,=\,m$. When contracted with any ordered family of $m$ vector fields (need not be 
distinct) chosen from $\{X_1,\, \cdots,\, X_d\}$, the tensor $t$ produces a holomorphic function on $X$ which vanishes at $x_0$. 
Therefore all those contractions vanish identically on $X$.

Now assume that $0\,<\,k\,<\,m$. When contracted with any $k$ vector fields among $\{X_1,\, \cdots,\, X_d\}$, the tensor $t$
produces a holomorphic section of $(T^*X)^{\otimes m-k}$ vanishing at $x_0$. Since $m-k\, <\, m$, the induction
hypothesis on $m$ holds. So this section of $(T^*X)^{\otimes m-k}$ vanishes identically.

Since all those contractions vanishing identically, our tensor $t$ is a holomorphic section of $((TX/V)^*)^{\otimes m}$.

The fibers of $\pi$ being compact and connected, the tensor $t$ is a pull-back from $Y$. We 
get a nontrivial holomorphic section of $((TY)^*)^{\otimes m}$ which vanishes at $\pi(x_0) 
\,\in\, Y$. But, as before, this holomorphic tensor should be parallel with respect to any 
Ricci flat metric on $Y$ \cite{Be}: a contradiction.
\end{proof}

\begin{remark} Assume that $Y$ is a compact complex manifold such that any holomorphic 
geometric structure of affine type on it is locally homogeneous. The proof of Theorem 
\ref{fibration} shows that any compact complex torus holomorphic principal bundle $X$ over 
$Y$ shares the same property (i.e. any holomorphic geometric structure of affine type on X 
is locally homogeneous as well).
\end{remark}

Corollary \ref{main corollary} is a direct consequence of 
Theorem \ref{fibration} because of the following proposition.

\begin{proposition} \label{prop: weak version}
Let $X$ be a compact complex manifold whose canonical bundle is of finite order such that $X$
admits a holomorphic rigid geometric structure $\phi$. If $\phi$ 
is locally homogeneous, then the fundamental group of $X$ is infinite. 
\end{proposition}

\begin{proof}
Assume by contradiction that the fundamental group of $X$ is finite. Replacing $X$ by its universal cover 
and $\phi$ by its pull-back on the universal cover we assume that $X$ is simply connected. Notice that $K_X$ has
now become trivial.

Since $\phi$ is locally homogeneous, the Killing Lie algebra of $\phi$ is transitive on $X$. The extendibility result of 
local Killing fields on simply connected manifolds \cite{Am, No, DG, Gr} implies that the connected component of identity 
of ${\rm Aut}(X, \phi)$ acts transitively on $X$. It now follows that $X$ is a compact complex homogeneous manifold. Since 
$K_X$ is trivial, this implies that $X$ is a parallelizable manifold biholomorphic to a quotient of a connected complex Lie 
group by a co-compact lattice in it \cite{Wa}. In particular, the fundamental group of $X$ is infinite: a contradiction.
\end{proof}

\begin{proposition}\label{dim less}
Let $X$ be a compact complex manifold in the Fujiki class $\mathcal{C}$, of complex 
dimension $n$ and of algebraic dimension $n-d$, with $d>0$. Assume that the canonical bundle $K_X$ of $X$ is trivial and 
that $X$ admits a holomorphic rigid geometric structure $\phi$. Then the fundamental group of $X$ is infinite. 
\end{proposition}

\begin{proof}
We assume that $X$ have complex dimension $n$ and algebraic dimension $n-d$, with $d\,>\,0$. Assume, by 
contradiction, that $X$ have finite fundamental group. Replacing $X$ by its universal cover and pulling-back $\phi$, we 
can assume that $X$ is compact simply connected. We have seen in the proof of Theorem \ref{act alg red} that the maximal 
connected abelian complex Lie subgroup $L'$ of the group of automorphisms of $(X, \,\phi)$ acts transitively on the generic 
fibers of the algebraic reduction (in particular $L'$ has positive dimension) and coincides with the automorphism group of 
the new holomorphic rigid geometric structure $\phi'$ constructed by juxtaposing $\phi$ with a basis
$\{X_1,\, \cdots,\, X_l\}$ of the Lie algebra of ${\rm Aut}(X, \phi)$ (this construction was used
earlier in the proof of Theorem \ref{act alg red}).
Moreover, since $L'$ also preserves the smooth finite volume defined by a nontrivial 
holomorphic section of $K_X$ (Lemma \ref{vol}), Section 3.7 in \cite{Gr} shows that 
the orbits of $L'$ are all compact (see also Section 3.5.4 in \cite{DG}). The orbits must be compact complex tori covered by $L'$.

The manifold $X$ being in class $\mathcal{C}$, the main theorem in \cite{GW} asserts that the action of $L'$ factors 
through a holomorphic action of a compact complex torus $T$. In particular, generic orbits of the $T$-action have trivial 
stabilizer, while all orbits have discrete stabilizers \cite{GW} (Lemma 2.1), and $X$ is a holomorphic principal $T$-Seifert 
bundle.

Now notice that $X$ being simply connected, its Albanese map is trivial \cite{Ue}, and the proof of Proposition 6.9 in 
\cite{Fu} asserts that any solvable Lie subalgebra of the Lie algebra of holomorphic vector fields on $X$ have a common 
zero on any invariant closed analytic set (see also Proposition 1.3 in \cite{GW}). In particular, (holomorphic) 
fundamental vector fields of the $T$-action should vanish identically on any compact complex torus embedded in $X$. This 
means that they should vanish identically on all $T$-orbits: a contradiction.
\end{proof}

We now deduce Corollary \ref{second corollary} from Proposition \ref{dim less}.

\begin{proof}[{Proof of Corollary \ref{second corollary}}] 
Let us consider a holomorphic affine connection in $TX$. We apply a method of \cite{At} (see also 
\cite{IKO}) to show that all Chern classes of $TX$ of positive degree must vanish. By Chern-Weil theory we compute a representative of the 
Chern class $c_k(X, \mathbb R)$ using a Hermitian metric on $TX$ and the associated Levi-Civita connection. We get a 
representative of $c_k(X, \mathbb R)$ which is a closed form on $X$ of type $(k,k)$. We perform the same computations using the 
holomorphic affine connection in $TX$ and get another representative of $c_k(X, \mathbb R)$ which is a holomorphic form 
on $X$. But on manifolds of type $\mathcal C$, forms of different types (here $(k,k)$ and $(2k,0)$) are never 
cohomologous, unless they represent zero in cohomology. This implies the vanishing of $c_k(X, \mathbb R)$, for all $k\, >\, 0$.

In particular, $c_1(X, \mathbb R)\,=\,0$ and Theorem 1.5 in \cite{To} implies that there exists a finite integer $l$ such that 
$K_X^l$ is holomorphically trivial.

Assume, by contradiction, that the fundamental group of $X$ is finite. We replace $X$ by its 
universal cover which is a compact complex manifold in class $\mathcal C$ with trivial 
canonical bundle bearing a holomorphic affine connection. Proposition \ref{dim less} implies 
that $X$ is a Moishezon manifold. By Corollary 2 in \cite{BM2}, a Moishezon manifold $X$ 
admitting a holomorphic Cartan geometry (in particular a holomorphic affine connection) must 
be a smooth complex projective manifolds. But a compact complex projective manifold bearing a 
holomorphic affine connection (and hence having trivial real Chern classes of positive 
degree) is covered by a compact complex torus \cite{IKO}: a contradiction.
\end{proof}

\section{Fujiki class $\mathcal C$ and geometric structures}\label{seF}

This section deals with holomorphic geometric structures on Fujiki class $\mathcal C$ manifolds with finite fundamental 
group. The main results are Theorem \ref{dim less 2}, Theorem \ref{Cartan geom} and Corollary \ref{third corollary}. 

Let us start by proving the following weak version of Theorem \ref{dim less 2}:

\begin{proposition}\label{n-1}
Let $X$ be a compact simply connected K\"ahler manifold, of complex dimension $n$ and of 
algebraic dimension $n-1$. Then $X$ does not admit any holomorphic rigid geometric structure.
\end{proposition}

\begin{proof} Assume, by contradiction, that $X$ as in the proposition bears a holomorphic rigid geometric structure 
$\phi$. Then Theorem \ref{act alg red} implies that there exists a holomorphic vector field $K$ on $X$ preserving the fibers 
of the algebraic reduction $\pi_{red}$ of $X$. The algebraic reduction $\pi_{red}$ of any compact K\"ahler manifold of 
algebraic dimension $n-1$ is known to be an almost holomorphic fibration. Then Theorem \ref{act alg red} implies that all 
$K$-orbits in $X$ are compact.

By a Theorem of Holman, \cite{Ho}, the $K$-action factors through the action of a compact complex torus $T$ of dimension 
one, and $X$ is a holomorphic Seifert $T$-principal bundle. It can be shown that the $T$-action must have trivial 
stabilizers on all of $X$. Indeed, if an element $g \,\in\, T$ fixes $x \,\in\, X$, then its differential at $x$ preserves 
$K$ and also acts trivially on the quotient $TX/ {\mathbb R}\cdot K$ (since the action fixes all fibers of the algebraic 
reduction). On the other hand, since $T$ is compact, the action of $g$ must be linearizable in the neighborhood of $x$. 
This implies that the action of $g$ is trivial in the neighborhood of $x$ and hence on $X$. It follows that $g$ is the 
identity element in $T$, and hence the $T$-action is free.

This implies that $X$ is a $T$-principal bundle over a simply connected projective manifold (the basis $V$ of the 
algebraic reduction). By a result of Blanchard \cite{Bl} (see also Theorem 1.6 in \cite{Hof}) those manifolds are K\"ahler if and 
only if the principal bundle is trivial. In particular, they are not simply connected if they are not
K\"ahler: a contradiction.
\end{proof}

Notice that the above proof of Proposition \ref{n-1} works for K\"ahler manifolds $X$ for which the algebraic reduction 
$\pi_{red}$ is a nontrivial almost holomorphic map. Another point regarding the previous proof is that it adapts to 
holomorphic Cartan geometries of algebraic type. Indeed, one needs to apply Theorem 1.2 in \cite{D2} in order to show that 
there exists a nontrivial holomorphic vector field $K$ preserving the Cartan geometry as well as the fibers of the 
algebraic reduction. Therefore, the proof of Proposition \ref{n-1} remains valid for holomorphic Cartan geometries of 
algebraic type.

\begin{theorem}\label{dim less 2}
Let $X$ be a compact complex manifold in the Fujiki class $\mathcal{C}$, of complex 
dimension $n$ and of algebraic dimension $n-d$, with $d\,>\,0$. If $X$ admits a holomorphic rigid geometric structure $\phi$, 
then the fundamental group of $X$ is infinite.
\end{theorem}

\begin{proof} Assume, by contradiction, that $X$ has finite fundamental group. We replace $X$ by its universal cover and 
the rigid geometric structure by the pull-back of $\phi$ on the universal cover. In this way we may assume that $X$ is 
simply connected. By Theorem \ref{act alg red}, the action of the connected complex abelian Lie group $L$ preserves the 
algebraic reduction $\pi_{red} \,:\, X \,\longrightarrow\, V$. Since $X$ is in class $\mathcal{C}$, this $L$ is a subgroup 
of the complex linear algebraic group ${\rm Aut}_0(X)$ (Corollary 5.8 in \cite{Fu}).

It follows that in Theorem \ref{thue} the model $\widetilde{X}$ may be chosen, using Hironaka's equivariant resolution 
with respect to ${\rm Aut}_0(X)$ (see Lemma 2.4 point (4), Remark 2.4 point (2) and Lemma 2.5 in \cite{Fu}), such that ${\rm 
Aut}_0(\widetilde{X})$ acts transitively, at the generic point in $\widetilde{X}$, on the fibers of the map $$t \,:\, 
\widetilde{X}\,\longrightarrow\, V\, .$$ Equivalently, , for any generic point $x_0 \,\in\, \widetilde{X}$, there exist 
holomorphic vector fields $\{X_1,\, \cdots,\, X_{d}\}$ on $\widetilde{X}$ such that $\{X_1(x_0),\, \cdots,\, X_{d}(x_0)\}$ span the 
tangent space at $x_0$ to the corresponding fiber $t^{-1}(t(x_0)).$ Since the fiber $t^{-1}(t(x_0))$ is compact and 
connected, by Blanchard Theorem, $\{X_1,\, \cdots, \,X_{d}\}$ span the tangent space to $t^{-1}(t(x_0))$ on an open dense set in 
$t^{-1}(t(x_0))$. It now follows that the Lie algebra of the stabilizer of $t^{-1}(t(x_0))$ in ${\rm Aut}_0(\widetilde{X})$
contains $\{X_1,\, \cdots,\, X_d\}$. Therefore, this stabilizer of $t^{-1}(t(x_0))$ 
acts with an open dense orbit in $t^{-1}(t(x_0))$. Consequently, 
$t^{-1}(t(x_0))$ is an almost homogeneous space (the group is the stabilizer in ${\rm Aut}_0(\widetilde{X})$).

The closure of any ${\rm Aut}_0(\widetilde{X})$-orbit is a compact analytic subset in $\widetilde{X}$ (\cite{Fu}, Lemma 
2.4 point (2)). Recall that $\widetilde{X}$ being simply connected and in Fujiki class $\mathcal C$, this ${\rm 
Aut}_0(\widetilde{X})$ is a complex linear algebraic group acting meromorphically on $X$ \cite{Fu} (Lemma 1 and Corollary 
5.8). This means --- as in the proof of Theorem \ref{thue} (ii) --- that the action of ${\rm Aut}_0(\widetilde{X})$ on 
$\widetilde{X}$ extends to a meromorphic map defined on a rational manifold. In particular, the closure of any ${\rm 
Aut}_0(\widetilde{X})$-orbit is the meromorphic image of a rational manifold and hence it is Moishezon.

The smooth fiber $t^{-1}(t(x_0))$ of $t$ is a closed analytic subset in the closure of the ${\rm 
Aut}_0(\widetilde{X})$-orbit of $x_0$. Consequently, $t^{-1}(t(x_0))$ is also Moishezon.

It can be shown that the smooth fiber $t^{-1}(t(x_0))$ does not admit any nontrivial holomorphic one-form. Indeed, any 
holomorphic one-form $\beta$ on it would furnish a constant function when specialized on any of the vector fields $\{X_1, 
\,\cdots,\, X_{d}\}$. Since all those vector fields vanish at some point in $t^{-1}(t(x_0))$ (Proposition 1.3 in \cite{GW} 
and the proof of Proposition 6.9 in \cite{Fu}), it follows that $\beta\,=\, 0$.

Since $t^{-1}(t(x_0))$ does not admit any nontrivial holomorphic one-form,
the first Betti number of the smooth 
fiber of $t$ is trivial. Then the results of Campana \cite[Corollary 2]{Ca} and Fujiki (Proposition 2.5 in \cite{Fu2},
see also \cite{Fu3}) imply that $\widetilde{X}$ is Moishezon, and hence $X$ is Moishezon: a contradiction.
\end{proof}

\begin{theorem} \label{Cartan geom} Let $X$ be a compact complex simply connected manifold in the Fujiki class 
$\mathcal{C}$. If $X$ bears a holomorphic Cartan geometry of algebraic type, then $X$ is a complex
projective variety. \end{theorem}

\begin{proof}
By Corollary 2 in \cite{BM2}, a Moishezon manifold $X$ admitting a holomorphic Cartan geometry must be
a smooth complex projective manifolds.

We deal now with the case where the algebraic dimension of $X$ is strictly less than the complex dimension of $X$.
In view of Theorem 1.2 in \cite{D2}, the 
method used in the proof of Theorem \ref{dim less 2} works for holomorphic Cartan geometries of algebraic type,
showing that compact simply connected manifolds in class $\mathcal C$ which are not of maximal 
algebraic dimension (meaning not Moishezon) do not admit any holomorphic Cartan geometry of algebraic type.
\end{proof}

Let us deduce Corollary \ref{third corollary} from Theorem \ref{Cartan geom}.

\begin{proof}[{Proof of Corollary \ref{third corollary}}] Assume, by contradiction, that the 
fundamental group of $X$ is finite. Replacing $X$ by its universal cover endowed with the 
pull-back of the Cartan geometry, enables one to assume that $X$ is simply connected. Theorem 
\ref{Cartan geom} implies than that $X$ is a projective manifold. With our assumption on the 
first Chern class, $X$ is a projective Calabi-Yau manifold. But a projective Calabi-Yau 
manifold bearing a holomorphic Cartan geometry of algebraic type is covered by a compact 
complex torus \cite{BM1,D2}: a contradiction.
\end{proof}

\section*{Acknowledgements}

We are very grateful to the referee for pointing out an error in a previous 
version; that part is now removed. We are grateful to Fr\'ed\'eric Campana for very 
helpful comments. We also thank Henri Guenancia for very valuable discussions. The 
first-named author is supported by a J. C. Bose Fellowship. The second-named author wishes to 
thank T.I.F.R. Mumbai for its hospitality.


\end{document}